\documentclass[a4wide,12pt]{article}
\usepackage{amsthm}
\usepackage{amssymb}
\usepackage{graphics}
\usepackage{gastex}
\usepackage[usenames]{color}

\newtheorem{theorem}{Theorem}
\newtheorem{lemma}[theorem]{Lemma}
\newtheorem{conject}[theorem]{Conjecture}
\newtheorem{prop}[theorem]{Proposition}

\renewcommand{\leq}{\leqslant}
\renewcommand{\geq}{\geqslant}

\newcommand\alert[1]{{\bf #1}}

\begin{document}

\title{Towards a Better Understanding of the Semigroup Tree}
\author{Maria Bras-Amor\'os\thanks{maria.bras@urv.cat, Departament d'Enginyeria Inform\`atica i Matem\`atiques, Universitat Rovira i Virgili, Av. Pa\"\i sos Catalans 26, 43007 Tarragona, Catalonia}
\ and Stanislav Bulygin \thanks{bulygin@mathematik.uni-kl.de, Department of Mathematics,
University of Kaiserslautern, P.O. Box 3049, 67653 Kaiserslautern, Germany}}

\maketitle

\begin{abstract}  
  In this paper we elaborate on the structure of the semigroup tree and the regularities on the number of descendants of each node observed in
  \cite{Bras:ngbounds}. These regularites admit two different types of behavior and in this work 
  we investigate which of the two types takes place in particular for well-known classes of semigroups. Also we study the question of what kind  
  of chains appear in the tree and characterize the properties (like being (in)finite) thereof. 
  We conclude with some thoughts that show how this study of the semigroup tree may help
  in solving the conjecture of Fibonacci-like behavior of the number of semigroups with given genus.
\end{abstract}

\noindent {\bf Keywords}: Numerical semigroup, Fibonacci numbers.

\section{Introduction}

A numerical semigroup is a subset of the non-negative integers ${\mathbb N}_0$
which is closed under addition, contains $0$ and its complement in
${\mathbb N}_0$ is finite.
The elements in this complement are called {\it gaps} and the number of gaps of a numerical semigroup is its {\it genus}. 
The smallest integer in a numerical semigroup from which all 
larger integers belong to the numerical semigroup is called the 
{\it conductor} of the numerical semigroup. Notice that the conductor 
of a numerical semigroup is exactly the largest gap 
(known as its {\it Frobenius number}) plus one.

It can be shown that each numerical semigroup has 
a unique minimal set of generators.
The numerical semigroups of genus $g$ can be obtained 
from the numerical semigroups of genus $g-1$ 
by taking out one by one the generators that are 
larger than or equal to the conductor of each semigroup. 
This leads to an infinite tree containing all numerical semigroups,
with root corresponding to the trivial semigroup
and where each level of nodes represents numerical semigroups of genus given by the level.
The parent of a numerical semigroup is obtained by adding to the semigroup
its Frobenius number.
This tree is illustrated in Figure~\ref{fig:tree},
where we used $\langle a_1,\dots,a_k\rangle$ to denote the numerical semigroup
generated by $a_1,\dots,a_k$.
 This construction was already considered in \cite{Rosales:families,RoGaGaJi:fundamentalgaps,RoGaGaJi:oversemigroups}.

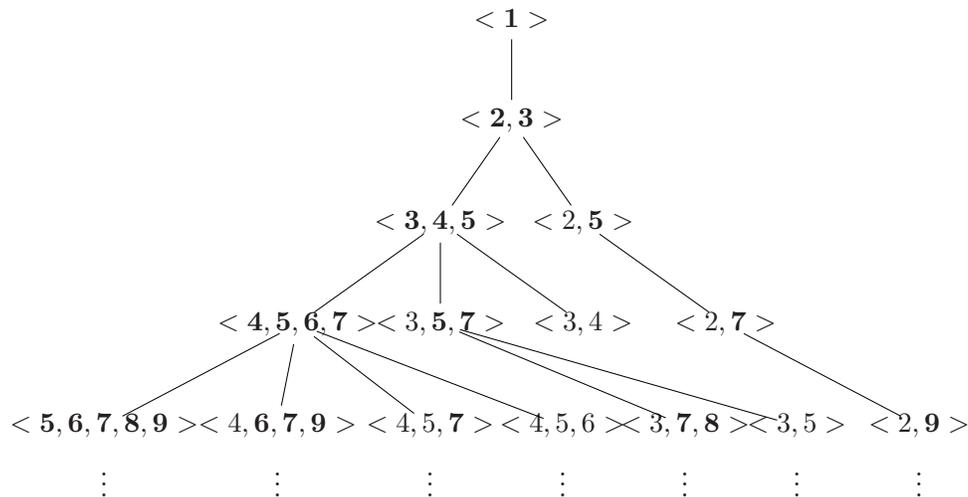
\begin{figure}
\compatiblegastexun
\setvertexdiam{4}
\letvertex S1=(40,50)
\letvertex S23=(40,40)
\letvertex S345=(33,30)
\letvertex S25=(47,30)
\letvertex S4567=(19,20)
\letvertex S357=(33,20)
\letvertex S34=(47,20)
\letvertex S27=(61,20)
\letvertex S56789=(0,10)
\letvertex P56789=(0,5)
\letvertex S4679=(17,10)
\letvertex P4679=(17,5)
\letvertex S457=(32,10)
\letvertex P457=(32,5)
\letvertex S456=(45,10)
\letvertex P456=(45,5)
\letvertex S378=(57,10)
\letvertex P378=(57,5)
\letvertex S35=(68,10)
\letvertex P35=(68,5)
\letvertex S29=(80,10)
\letvertex P29=(80,5)
\begin{center}
\resizebox{.8\textwidth}{!}{
{\small
\begin{picture}(80,50)
\drawvertex(S1){$<{\bf 1}>$}
\drawvertex(S23){$<{\bf 2},{\bf 3}>$}
\drawvertex(S345){$<{\bf 3},{\bf 4},{\bf 5}>$}
\drawvertex(S25){$<2,{\bf 5}>$}
\drawvertex(S4567){$<{\bf 4},{\bf 5},{\bf 6},{\bf 7}>$}
\drawvertex(S357){$<3,{\bf 5},{\bf 7}>$}
\drawvertex(S34){$<3,4>$}
\drawvertex(S27){$<2,{\bf 7}>$}
\drawvertex(S56789){$<{\bf 5},{\bf 6},{\bf 7},{\bf 8},{\bf 9}>$}
\drawvertex(P56789){$\vdots$}
\drawvertex(S4679){$<4,{\bf 6},{\bf 7},{\bf 9}>$}
\drawvertex(P4679){$\vdots$}
\drawvertex(S457){$<4,5,{\bf 7}>$}
\drawvertex(P457){$\vdots$}
\drawvertex(S456){$<4,5,6>$}
\drawvertex(P456){$\vdots$}
\drawvertex(S378){$<3,{\bf 7},{\bf 8}>$}
\drawvertex(P378){$\vdots$}
\drawvertex(S35){$<3,5>$}
\drawvertex(P35){$\vdots$}
\drawvertex(S29){$<2,{\bf 9}>$}
\drawvertex(P29){$\vdots$}
\drawundirectededge(S1,S23){}
\drawundirectededge(S23,S345){}
\drawundirectededge(S23,S25){}
\drawundirectededge(S345,S4567){}
\drawundirectededge(S345,S357){}
\drawundirectededge(S345,S34){}
\drawundirectededge(S25,S27){}
\drawundirectededge(S4567,S56789){}
\drawundirectededge(S4567,S4679){}
\drawundirectededge(S4567,S457){}
\drawundirectededge(S4567,S456){}
\drawundirectededge(S357,S378){}
\drawundirectededge(S357,S35){}
\drawundirectededge(S27,S29){}
\end{picture}}}\end{center}
\caption{Recursive construction of the numerical semigroups of genus $g$ from the numerical
semigroups of genus $g-1$. Generators larger than the conductor are written in bold face.}
\label{fig:tree}
\end{figure}

The number $n_g$ of all numerical semigroups of genus $g$ has been studied in \cite{Bras:Fibonacci,Bras:ngbounds}. 
In \cite{Bras:Fibonacci} it is conjectured that $n_g$ asymptotically behaves 
like the Fibonacci numbers.
That is, $n_g\geq n_{g-1}+n_{g-2}$,
$\lim_{g\to\infty}(n_{g-1}+n_{g-2})/n_g=1$,
and  
$n_g/n_{g-1}$ approaches the golden ratio.
In \cite{Bras:ngbounds} 
the tree of numerical semigroups is used to derive, for $g\geq 3$, the bounds
$2F_{g}\leq n_g\leq 1+3\cdot 2^{g-3}$, where $F_g$ denotes the $g$-th Fibonacci number.
The goal of this paper is providing results for better understanding the
semigroup tree and giving possible directions for attacking 
the previous conjecture.
The bounds given in \cite{Bras:ngbounds} are a consequence of
the fact that only two kinds of generators exist in a numerical semigroup
larger than or equal to its conductor.
In Section~\ref{sec:knownclasses} we call these two kinds of generators
{\em weak} and {\em strong} and we study 
their existence in three well-known classes of numerical semigroups:
symmetric, pseudo-symmetric, and Arf semigroups.

In Section~\ref{sec:infinitechains} we 
analyze which nodes have an infinite number of descendants.
For the nodes having a finite number of descendants 
we give a way to determine the descendant at largest distance;
for the nodes having an infinite number of descendants
we determine the number of infinite chains in which the semigroup lies.
It turns out here that 
primality and coprimality of integers appear in the scene as discriminating
factors. Some results related to weak and strong generators of semigroups
lying in infinite chains are also given.

In the last section we give what we think should be future directions 
for attacking the conjecture on the Fibonacci-like behavior of $n_g$
and how the results presented in the first sections could help.

\section{Behavior of known classes of numerical semigroups}
\label{sec:knownclasses}

The enumeration $\lambda$ of a numerical semigroup $\Lambda$ is the unique increasing bijective map ${\mathbb N}_0\rightarrow\Lambda$. Usually 
$\lambda(i)$ is denoted $\lambda_i$.
It is easy to check that
if $c$ and $g$ are the conductor and the genus of $\Lambda$ then
$\lambda_{c-g}=c$ and for
$\lambda_i\geq c$, $\lambda_i=i+g$.
A semigroup for which $\lambda_1=c$,
i.e. a semigroup of the form $\{0\}\cup[c,\infty)$,
is called \emph{ordinary}.

It was shown in \cite{Bras:ngbounds} that the next Lemma holds.

\begin{lemma}
\label{lemma:reasonforweakstrong}
If $\lambda_{i_1}<\lambda_{i_2}<\dots<\lambda_{i_n}$ 
are the generators of a non-ordinary numerical semigroup $\Lambda$
that are larger than or equal to its conductor then 
the generators of $\Lambda\setminus\{\lambda_{i_j}\}$ that are larger than 
or equal to its conductor
are either $\lambda_{i_{j+1}},\dots,\lambda_{i_n}$
or $\lambda_{i_{j+1}},\dots,\lambda_{i_n},\lambda_{i_j}+\lambda_1$. 
\end{lemma}

Motivated by this lemma, we call the generators of a non-ordinary
numerical semigroup 
that are larger than or equal to its conductor, the \emph{effective generators}
and we say that an effective generator $\lambda_{i_j}$ is \emph{strong} if the set of effective generators of 
$\Lambda\setminus\{\lambda_{i_j}\}$ is $\lambda_{i_{j+1}},\dots,\lambda_{i_n},\lambda_{i_j}+\lambda_1$.
An effective generator that is not strong is called a \emph{weak generator}.

Finally we say that
  a \emph{leave} is a node with no descendants, a \emph{stick} is a node with exactly one descendant and  a \emph{bush} is a node with
  two or more descendants.

\subsection{Symmetric semigroups}

{\em Symmetric semigroups} are those
semigroups for which the conductor is twice the genus. Symmetric semigroups 
and their applications to coding theory
have been 
studied, among others, in \cite{BaDoFo,CaFa:SingularPlaneModels,HoLiPe,KiPe:telescopic}. 
An important property of symmetric semigroups
is that if $c$ and $g$ are the genus and the conductor of a symmetric semigroup
$\Lambda$ then any integer $i$ is a gap of $\Lambda$ if and only if $c-1-i$ is a non-gap.

The semigroups of the form $\langle 2,2n+1\rangle, n\ge 1$ are symmetric. 
They are called {\em hyperelliptic semigroups}.

\begin{lemma}
  Hyperelliptic numerical semigroups
  are sticks and the unique 
  effective generator, which is the conductor plus one, is strong.
  \label{hyper}
\end{lemma}

Given a numerical semigroup $\Lambda$ with enumeration $\lambda$,
the associated $\nu$-sequence is defined by 
$\nu_i=\#\{j\in{\mathbb N}_0:\lambda_i-\lambda_j\in\Lambda\}$.
It is proven in
\cite[Theorem 3.8]{KiPe:telescopic}
that 
\begin{equation}
  \label{eq:D(i)}
  \nu_i=i-g(i)+\#D(i)+1,
\end{equation} 
where $g(i)$ is the number of gaps smaller than $\lambda_i$,
and $D(i)=\{l\not\in\Lambda|\lambda_i-l\not\in\Lambda\}$.
Notice that an element $\lambda_i\in\Lambda$ is a generator of $\Lambda$ if and only if $\nu_i=2$.

\begin{lemma}
\label{lemma:D(i)}
For a numerical semigroup with enumeration $\lambda$ and conductor $c$,
an element $\lambda_i\geq c$ is a generator
if and only if $$\#D(i)=g-i+1.$$
\end{lemma}

\begin{proof}
It follows from equality~(\ref{eq:D(i)})
and from the fact that a non-gap $\lambda_i$ is
a generator if and only if $\nu_i=2$. 
\end{proof}

\begin{lemma}
  Non-hyperelliptic symmetric semigroups are leaves.
  \label{symm-leave}
\end{lemma}

\begin{proof}
For a symmetric semigroup with conductor $c$ and genus $g$,
$\lambda_i\geq c$ if and only if $i\geq g$.
Hence, by Lemma~\ref{lemma:D(i)},
$\lambda_i\geq c$ can only be a generator if
$\lambda_i=c$ and $\#D(i)=1$
or if $\lambda_i=c+1$ and $\#D(i)=0$.
The first situation is only possible when $1=c-1$ because $1,c-1\not\in\Lambda, 1+(c-1)=c$ and otherwise,
$\#D(i)>1$. But $1=c-1$ would mean that $c=2$ and thus
the numerical semigroup would be hyperelliptic.
The second situation is only possible for hyperelliptic semigroups since for other semigroups,
$2$ and $c-1$ are gaps and $2+(c-1)=c+1$. This implies $\#D(i)>0$.
\end{proof}

As an example of non-hyperelliptic symmetric semigroup
consider $\Lambda=\{0,4,5,$ $8,9,10\}\cup[12,\infty)$.
In this case the generators are $4$ and $5$
and none of them is effective.

\subsection{Pseudo-symmetric semigroups}

{\em Pseudo-symmetric semigroups} are those
semigroups for which the conductor is twice the genus minus one. 
An important property of pseudo-symmetric semigroups
analogous to the one for symmetric semigroups is that
if $c$ and $g$ are the genus and the conductor of a pseudo-symmetric semigroup
$\Lambda$ then any integer $i$ different from $(c-1)/2$ is a gap of $\Lambda$ if and only if $c-1-i$ is a non-gap.

\begin{lemma}
  For a non-ordinary numerical semigroup $\Lambda$ with enumeration $\lambda$ and conductor $c$,
  a non-gap $\lambda_k\neq 2\lambda_1$ is a strong generator if and only if $\lambda_k\geq c$
  and $\nu_{k+\lambda_1}=4$.
  \label{strong-characterize}
\end{lemma}

\begin{proof}
If $\lambda_k$ is strong, then by definition $\lambda_k\geq c$. Now, $\nu_{k+\lambda_1}\geq 4$
because $\lambda_{k+\lambda_1}-\lambda_1=\lambda_k$, 
$\lambda_{k+\lambda_1}-\lambda_k=\lambda_1$,
and $0,\lambda_1,\lambda_k,\lambda_{k+\lambda_1}$ are different.
If $\nu_{k+\lambda_1}>4$ this means that there exists at least one $\lambda_l$ with $l$ different
from $k$ such that $\lambda_{k+\lambda_1}-\lambda_l\in\Lambda$
and $\lambda_{k+\lambda_1}-\lambda_l\neq \lambda_k$.
Then $\lambda_{k+\lambda_1}$ is not a generator of $\Lambda\setminus\{\lambda_k\}$.

On the other hand, if $\lambda_k\geq c$
and $\nu_{k+\lambda_1}=4$ this means that $\lambda_k$ is a generator.
Indeed, if $\lambda_k=\lambda_l+\lambda_m$ with $0<l\leq m<k$ then
$\lambda_{k+\lambda_1}-\lambda_l=\lambda_m+\lambda_1\in\Lambda$, so, 
$\nu_{k+\lambda_1}>4$. Furthermore, since $\nu_{k+\lambda_1}=4$ this means that $\lambda_k+\lambda_1$ can only be subtracted by 
$0,\lambda_1,\lambda_k,\lambda_{k+\lambda_1}$ within the numerical semigroup. Consequently,
$\lambda_k+\lambda_1$ is a generator of $\Lambda\setminus\{\lambda_k\}$.
\end{proof}

\begin{lemma}
\label{lemma:357}
\begin{enumerate}
\item
The unique pseudo-symmetric semigroup of genus $g$ 
with only one interval of non-gaps 
between $0$ and the conductor is
$\Lambda_{{ps}_{g}}=\{0,g,g+1,\dots,2g-3\}\cup[2g-1,\infty)$.
\item
  The numerical semigroup
  $\Lambda_{{ps}_3}
=\{0,3\}\cup[5,\infty)$, has 5 and 7 as the only effective generators.
  The generator $5$ is strong and the generator $7$ is weak.
\item
  The numerical semigroup
  $\Lambda_{{ps}_4}
=\{0,4,5\}\cup[7,\infty)$, has 7 as the only effective generator and it is strong.
\item
The numerical semigroup
  $\Lambda_{{ps}_g}$, for $g\geq 5$ is a stick, its 
unique effective generator is $c$, and it is weak.
\end{enumerate}
\end{lemma}

\begin{proof}
The proof of statement 1 follows directly from the main property of pseudosymmetric semigroups. 
Statements 2 and 3 can be proved 
by an exhaustive search of generators and 
by checking which are weak and which are strong.

Since the conductor of $\Lambda_{{ps}_g}$
is $2g-1$, every integer larger than or equal to 
$4g-2$ will not be a generator.
The integer $4g-3$ is not a generator since $4g-3=g+(2g-3)$.
The integer $4g-4$ is not a generator since $4g-4=(2g-3)+(2g-1)$.
The integers from $2g$ to $4g-6$ are generated by the interval
$g,\dots,2g-3$.
So the only effective generator of 
$\Lambda_{{ps}_g}$ can be $c=2g-1$ and $4g-5$.
It is easy to check that $c$ is a generator. 
If the integer $4g-5$ is 
larger than or equal to
$g+(2g-1)$ then it is not a generator. This is equivalent to $g\geq 4$.

On the other hand,
$2g-1$ is weak if and only if $g+(2g-1)$ is a sum of two non-gaps strictly
smaller than $2g-1$ and this is equivalent to having  
$g+g\leq g+(2g-1)\leq (2g-3)+(2g-3)$, which in turn is equivalent to
$g\geq 5$.
Thus
$c$ is a weak generator if $g\geq 5$.
\end{proof}

\begin{lemma}
\label{lemma:pseudo-symmetrics_multiplicity_3}
\begin{enumerate}
\item A numerical semigroup is pseudo-symmetric and has $\lambda_1=3$ if and only if 
it is equal to
$\Lambda=\{0,3,6,\dots,3k,3(k+1)-1,3(k+1),3(k+2)-1,3(k+2),\dots, 3(2k-1)-1,3(2k-1)\}\cup[3(2k-1)+2,\infty)$
 or  
$\Lambda=\{0,3,6,\dots,3k,3(k+1),3(k+1)+1,3(k+2),3(k+2)+1,\dots, 3(2k),3(2k)+1\}\cup[3(2k)+3,\infty)$ for some $k$.
\item 
Each pseudo-symmetric semigroup with $\lambda_1=3$
has a unique effective generator, it is $c+2$
and it is weak.
\item 
The descendants of a
pseudo-symmetric semigroups with $\lambda_1=3$
are  
non-hyperelliptic symmetric semigroups, and thus, leaves.
\end{enumerate}
\end{lemma}

\begin{proof}
\begin{enumerate}
\item 
From the property of pseudo-symmetric semigroups
that any non-negative integer $i$ different from $(c-1)/2$ is a gap if and only if 
$c-1-i$ is a non-gap we deduce that each pseudo-symmetric semigroup with $\lambda_1=3$
must be one of the semigroups above.
To see that these semigroups are always pseudo-symmetric,
let us compute the genus and the conductor. In the first case we have that up to $3k$ the semigroup $\Lambda$ has 
  exactly $2k$ gaps: 2 gaps per interval $[3i,3i+2], 0\le i\le k-1$. Then from $3k+1$ to $3(2k-1)$ there are $k-1$ gaps: one per interval
  $[3i+1,3(i+1)], k\le i\le 2k-2$. Together with the gap $3(2k-1)+1$ that makes $g=2k+k-1+1=3k$ gaps. Obviously $c=3(2k-1)+2=2g-1$. So
  $\Lambda$ is pseudo-symmetric. The other case is done analogously and we have $g=2(k+1)+k=3k+2$, $c=3(2k)+3=2g-1$.
\item
An element larger than or equal to the conductor must be $c+i=\lambda_{c-g+i}$ for some $i\geq 0$.
Since these semigroups are pseudo-symmetric,
$c-g+i=g+i-1$. 
Now, by Lemma~\ref{lemma:D(i)}, $\lambda_{g+i-1}$
is a generator if and only if $D(g+i-1)=2-i$.
Since $D(g+i-1)\geq 0$, this means that $i\leq 2$.
So, the only elements larger than or equal to the conductor that can be generators 
are $c,c+1,c+2$.
The elements $c$ and $c+1$ cannot be generators, because $c=3(2k-1)+2=3(2k-1)-1+3, c+1=3(2k-1)+3$ for the first case
and similarly is done for the second. Let us show that $c+2$ is a generator. Consider the first case, the second one is done analogously.
We have $c+2=3(2k-1)+4=6k+1$, so it has residue 1 modulo 3. Note that all the non-gaps less than $c+2$ have residues 0 or 2 modulo 3. So, if
$c+2$ is not a generator, it is a sum of two non-gaps with residue 2. So we have $c+2=3(k+i)-1+3(k+j)-1=6k+3i+3j-2$ for some
$i,j\ge 1$. But then we have that $i+j=1$, a contradiction.

To see that $c+2$ is a weak generator, suppose that $\lambda_k=c+2$. Since $\lambda_k>c$, then $\lambda_{k+\lambda_1}=c+5$, so 
$k+\lambda_1=c+5-g=g+4$. Assume that $c+2$ is a strong generator, then by (\ref{eq:D(i)}) we have
$\nu_{k+\lambda_1}=k+\lambda_1-g+\#D(k+\lambda_1)+1$ and by Lemma~\ref{strong-characterize}, we have $4=g+4-g+\#D(k+\lambda_1)+1$, 
  thus  $1\le\#D(k+\lambda_1)+1=0$, a 
  contradiction.
\item The only descendant of $\Lambda$ is obtained by removing $c+2$. 
The semigroup $\Lambda\setminus\{c+2\}$ is symmetric since its genus is $g+1$
and its conductor is $c+3$, and we have $c+3=2(g+1)$, since $c=2g-1$.
It is easy to see that $\Lambda\setminus\{c+2\}$ 
is 
  non-hyperelliptic semigroup, and thus a leave, cf. Lemma~\ref{symm-leave}.
\end{enumerate}
\end{proof}

\begin{lemma}
Each pseudo-symmetric semigroup with $\lambda_1\neq 3$ and with
more than one interval
of non-gaps between $0$ and the conductor
is a leave.
\end{lemma}

\begin{proof}
By Lemma~\ref{lemma:D(i)},
the only cases in which $\lambda_k\geq c$ can be a generator correspond to the next three situations.
We used that $\lambda_k\geq c$ if and only if $k\geq c-g$.
\begin{itemize}
\item 
$\lambda_k=c$ if $\#D(k)=2$.
Since there exists more than one interval of non-gaps between $0$ and $c$,
there exists $i\in\Lambda$, $i\neq 0, c-2$ such that $i+1\not\in\Lambda$.
So $c-1-i\not\in\Lambda$ (pseudo-symmetric property)
and $i+1$, $c-i-1$ are different from $1,c-1$ and they also add up to $c$. Hence, $\#D(k)>2$, a contradiction.
\item
$\lambda_k=c+1$ if $\#D(k)=1$. This case is impossible since $2\neq c-1$, both 2 and $c-1$ are 
gaps, and they add up to $c+1$.
\item
$\lambda_k=c+2$ if $\#D(k)=0$.
This is impossible if $\lambda_1\neq 3$ because $3$ and $c-1$ are then gaps and so 
$D(k)\neq\emptyset$.
\end{itemize}
\end{proof}

As an example of pseudo-symmetric semigroup
with $\lambda_1\neq 3$ and 
with more than one interval of non-gaps between $0$ and the conductor 
we can take $\Lambda=\{0,4,7,8,9\}\cup[11,\infty)$.
In this case the generators are $4,7,9$ and none of them is effective.

A numerical semigroup is said to be \emph{irreducible} if it cannot
be expressed as an intersection of two numerical semigroups properly containing it. 
It was proven in \cite{RoBr:irreducible}
that irreducible semigroups are exactly symmetric and pseudo-symmetric semigroups.
Thus we have 
shown that the only non-leaves corresponding to irreducible numerical semigroups 
are those treated in 
Lemmas \ref{hyper}, \ref{lemma:357}, \ref{lemma:pseudo-symmetrics_multiplicity_3}.
Moreover the number of effective generators is small and the number of 
strong generators is even smaller. Therefore, the parts of the semigroup tree in a vicinity of 
an irreducible semigroup 
are not "bushy" and are easily described.

%

\subsection{Arf semigroups}

A numerical semigroup $\Lambda$ with enumeration $\lambda$
is said to be Arf if $\lambda_i+\lambda_j-\lambda_k\in\Lambda$
for every $i,j,k\in{\mathbb N}_0$ with
$i\geq j\geq k$.
Hyperelliptic semigroups are an example of Arf semigroups.
In fact, it was shown in \cite{CaFaMu:arf}
that hyperelliptic semigroups are the only Arf symmetric semigroups.
A lot of work has been done related to Arf semigroups.
One can see, for instance,
\cite{BaDoFo,RoGaGaBr:arf,CaFaMu:arf}.

For the next lemma we use the fact that for an Arf numerical semigroup 
$\Lambda$, an element
$\lambda_i\neq 0,\lambda_1$ is a generator if and only if 
$\lambda_i-\lambda_1\not\in\Lambda$.

\begin{lemma}
\begin{enumerate}
\item
Non-hyperelliptic Arf numerical semigroups are bushes.
\item
Arf semigroups appear as descendants of semigroups with strong generators when removing one such generator.
\end{enumerate}
\end{lemma}

\begin{proof}
\begin{enumerate}
\item
For an Arf semigroup we know that if $i,i+1\in\Lambda$, then $i\ge c$. 
Indeed, for $j\geq i$, 
\begin{eqnarray*}
j&=&i+\overbrace{((i+1)-i)+((i+1)-i)+\dots+((i+1)-i)}^{(j-i)}\\
&=&\underbrace{\underbrace{\underbrace{i+((i+1)-i)}_{\in\Lambda}+((i+1)-i)}_{\in\Lambda}+\dots+((i+1)-i).}_{\in\Lambda}\\
\end{eqnarray*}
Thus we know that $c-1$ and either $c-2$ or $c-3$ are gaps.
Since $\Lambda$ is not hyperelliptic, $\lambda_1\geq 3$.
Thus, $c-1+\lambda_1$ and either $c-2+\lambda_1$
or $c-3+\lambda_1$ are generators.
\item
It follows from the remark previous to the Lemma. 
\end{enumerate}
\end{proof}

It was shown in \cite{RoGaGaBr:arf}
that at most two of the descendants of Arf semigroups are Arf.
For illustrating this,
notice that
$\{0,5,7\}\cup[9,\infty)$ 
has no Arf descendants;
$\{0,5\}\cup[7,\infty)$ 
has two Arf descendants: 
$\{0,5\}\cup[8,\infty)$
and 
$\{0,5,7\}\cup[9,\infty)$;
$\{0,5\}\cup[10,\infty)$
has one Arf descendant:
$\{0,5,10\}\cup[12,\infty)$.

\section{Infinite chains}
\label{sec:infinitechains}


We say that an infinite sequence 
of numerical semigroups
$\Lambda_0={\mathbb N}_0,\Lambda_1,\Lambda_2,\dots$ 
is an \emph{infinite chain} if for each $i\geq 1$, $\Lambda_{i-1}$
can be obtained by adding to $\Lambda_i$ its Frobenius number.
Clearly, a numerical semigroup has infinitely many descendants in the semigroup
tree if and only if
it lies in an infinite chain.

For the proof of the next lemma we will use that a set of integers
$l_1,\dots,l_m$ generate a numerical semigroup if and only if they are coprime.

\begin{lemma}
  Given an infinite chain $(\Lambda_i)_{i\geq 0}$,
  $$
    \bigcap_{i\geq 0}\Lambda_i=d\cdot \Lambda
  $$ 
  for some integer $d>1$ and some numerical semigroup $\Lambda$.
\end{lemma}

\begin{proof}
The intersection $\cap_{i\geq 0}\Lambda_i$
satisfies $0\in \cap_{i\geq 0}\Lambda_i$
and $x+y\in \cap_{i\geq 0}\Lambda_i$ for all $x,y\in \cap_{i\geq 0}\Lambda_i$.
Furthermore, all elements in $\cap_{i\geq 0}\Lambda_i$
must be divisible by an integer $d>1$. Indeed, otherwise
we could find a finite set of coprime elements 
which would generate a numerical semigroup,
and this numerical semigroup should be a subset of $\cap_{i\geq 0}\Lambda_i$.
Then the infinite chain would not contain any semigroup with 
genus larger than that of this semigroup, giving a contradiction.
Let $d$ be the greatest of the common divisors of $\cap_{i\geq 0}\Lambda_i$.
Then $\frac{1}{d}\left(\cap_{i\geq 0}\Lambda_i\right)$ must be a numerical semigroup.
\end{proof}

\begin{lemma}
  Given an integer $d>1$ and a numerical semigroup $\Lambda$
  the infinite chain obtained by deleting repetitions in the sequence
  $\Lambda_j=d\cdot \Lambda\cup\{l\in{\mathbb N}:l\geq j\}$ has intersection $d\cdot \Lambda$.
\end{lemma}

Consequently, if we denote by ${\mathbb S}$ the set of all numerical semigroups, there is a bijection
$${\mathbb S}\times{\mathbb N}_{\ge 2} \leftrightarrow \{\mbox{infinite chains}\}$$

In the next theorem we show that the greatest common divisor of the first elements
of a numerical semigroup determine whether the numerical semigroup
has infinite number of descendants. Notice that since $\lambda_{c-g}=c$,
the set $\lambda_0,\dots,\lambda_{c-g-1}$ is the set of non-gaps smaller than the conductor.

\begin{theorem}
\label{thm:infinitechains}
  Let $\Lambda$ 
  be a numerical semigroup with enumeration $\lambda$, genus $g$, and conductor $c$, 
and let 
  $d$ be the greatest common divisor of $\lambda_0,\dots,\lambda_{c-g-1}$.
  Then,
  \begin{enumerate}
  \item 
    $\Lambda$ lies in an infinite chain if and only if $d\neq 1$.
  \item 
    If $d=1$ then the descendant of $\Lambda$ with largest genus
    is the numerical semigroup generated by  
    $\lambda_0,\dots,\lambda_{c-g-1}$. 
  \item 
    If $d\neq 1$ then 
    $\Lambda$ lies in infinitely many infinite chains if and only if $d$ is not prime.
  \item If $d$ is a prime then 
    the number of infinite chains in which $\Lambda$ lies is
    the number of descendants of 
    $\{\frac{\lambda_0}{d},\frac{\lambda_{1}}{d},\dots,\frac{\lambda_{c-g-1}}{d}\}\cup
    \{l\in{\mathbb N}_0: l \geq \lceil\frac{c}{d}\rceil\}$. 
  \end{enumerate}
\end{theorem} 

\begin{proof}
\begin{enumerate}
\item
If $d=1$ then
$\lambda_0,\dots,\lambda_{c-g-1}$ 
generate a numerical semigroup $\Lambda'$
and each descendant of $\Lambda$ must contain $\Lambda'$.
Thus, the maximum of the genus of the descendants is the genus of $\Lambda'$ which is finite.
On the other hand, if $d\neq 1$ then
$$\lambda_0=d\tilde\lambda_0,\dots,\lambda_{c-g-1}=d\tilde\lambda_{c-g-1}$$ with 
$\tilde\lambda_0,\dots,\tilde\lambda_{c-g-1}$ coprime.
Let $\tilde\Lambda$ be the numerical semigroup generated by
$\tilde\lambda_0,\dots,\tilde\lambda_{c-g-1}$.
Consider the sequence of semigroups $$\Lambda_i=d\cdot\tilde\Lambda\cup\{l\in {\mathbb N}_0:l\geq i\}.$$ By deleting repetitions we obtain an 
infinite chain that contains $\Lambda$.
\item It follows from the proof of the previous statement.
\item If $d$ is not prime then $d=d_1d_2$ for some $d_1,d_2>1$ and, as before, 
$$
  \lambda_0=d_1d_2\tilde\lambda_0,\dots,\lambda_{c-g-1}=d_1d_2\tilde\lambda_{c-g-1}
$$ 
with $\tilde\lambda_0,\dots,\tilde\lambda_{c-g-1}$ coprime.
Let $\tilde\Lambda$ be the numerical semigroup generated by
$\tilde\lambda_0,\dots,\tilde\lambda_{c-g-1}$.
For each $i\geq 0$
and each $j\geq 0$
define $$\Lambda_{i,j}=d_1d_2\tilde\Lambda\cup \{d_1l\in{\mathbb N}_0:l\geq i\}\cup\{l\in{\mathbb N}_0:l\geq j\}.$$
For each fixed $i\geq \lceil\frac{c}{d_1}\rceil$, by deleting repetitions in the sequence 
$\left(\Lambda_{i,j}\right)_{j\geq 0}$ we obtain an infinite chain.
Moreover every such chain contains $\Lambda$, as $\Lambda=\Lambda_{i,c}$, 
if $i\ge \lceil\frac{c}{d_1}\rceil$. 
For each $i\geq \lceil\frac{c}{d_1}\rceil$
 this chain is different.
Thus we get infinitely many infinite chains.
The complete result in this statement follows from statement 4.
\item 
Suppose that an infinite chain $(\Lambda_i)_{i\geq 0}$ contains $\Lambda$.
It must satisfy 
$\cap_{i\geq 0}\Lambda_i=d\cdot\tilde{\Lambda}$ for a unique
numerical semigroup 
$\tilde{\Lambda}$
such that
\begin{itemize}
\item 
$d\tilde{\lambda}_0=\lambda_0,\dots
d\tilde{\lambda}_{c-g-1}={\lambda}_{c-g-1}$,
\item
$d\tilde{\lambda}_{c-g}\geq c$, since $d\tilde{\Lambda}\subseteq \Lambda$.
\end{itemize}
Thus, $\tilde{\Lambda}$ is a descendant of
$\{\frac{\lambda_0}{d},\frac{\lambda_{1}}{d},\dots,\frac{\lambda_{c-g-1}}{d}\}\cup
\{l\in{\mathbb N}_0: l \geq \lceil\frac{c}{d}\rceil\}$. 
\end{enumerate}
\end{proof}

\begin{lemma}
Let $\Lambda$ be a numerical semigroup with enumeration $\lambda$, genus $g$, conductor $c$, and
$\gcd(\lambda_0,\dots,\lambda_{c-g-1})=d>1$ lying in an infinite chain. Then 
\begin{enumerate}
\item All non-gaps between $c$ and $c+\lambda_1-1$ that 
are not multiples of $d$ are generators. Thus $\Lambda$ has at least $\lambda_1-\frac{\lambda_1}{d}$ effective generators.
\item If there are at least two non-gaps between $0$ and $c$, 
then all non-gaps between $c$ and $c+d-1$ that are not multiples of $d$ 
are strong generators. 
Thus $\Lambda$ has at least $d-1$ strong generators.
\item If there is just one non-gap between $0$ and $c$,
then there is at least one strong generator.
\end{enumerate}
\end{lemma}

\begin{proof}
\begin{enumerate}
\item If $c\leq \lambda_k\leq c+\lambda_1-1$,
$\lambda_k$ is not a multiple of $d$,
and there exist $0<i<j$ such that $\lambda_i+\lambda_j=\lambda_k$
then it must be $\lambda_j<c$; otherwise
$\lambda_k=\lambda_i+\lambda_j\geq \lambda_1+c$. But if $\lambda_i,\lambda_j<c$ then $\lambda_k=\lambda_i+\lambda_j$ 
is a multiple of~$d$, since $\lambda_i$ and $\lambda_j$ are, a contradiction.

\item If $c\leq \lambda_k\leq c+d-1$,
$\lambda_k$ is not a multiple of $d$, 
and there exist $1<i<j$ such that $\lambda_i+\lambda_j=\lambda_1+\lambda_k$
then it must be $\lambda_i<c$. Otherwise
$\lambda_2+c>\lambda_1+c+d-1\geq\lambda_1+\lambda_k=\lambda_i+\lambda_j\geq 2c$, a contradiction since $\lambda_2\leq c$. 
But then $\lambda_i$ is a multiple of $d$ and $\lambda_i+\lambda_j=\lambda_1+\lambda_k$ means that $\lambda_j\equiv\lambda_k\mbox{ mod }d$.
By hypothesis $\lambda_k$ is not a multiple of $d$ and so $\lambda_j$ is 
not a multiple of $d$ either and consequently $\lambda_j\geq c$.
But then $\lambda_k-\lambda_j\leq c+d-1-c=d-1$, so $\lambda_j=\lambda_k$
and $\lambda_i=\lambda_1+\lambda_k-\lambda_j=\lambda_1$, a contradiction.
\item For the last statement notice that at least $c$ or $c+1$ is strong. 

\end{enumerate}
\end{proof}

Notice that in the second statement of the previous lemma 
the requirement that
there are at least two non-gaps between $0$ and $c$
is necessary. As a counterexample consider the semigroup $\{0,8\}\cup[10,\infty)$.
In this case, 
$d=\lambda_1=8$ and all non-gaps between $10$ and $10+\lambda_1-1=17$ are generators except for $16$ which is a multiple of $d$. This is a 
consequence 
of the first statement.
The second statement fails since $12$ is between $c$ and $c+d-1$ and it is not a multiple of $d$, but $12+8=10+10$.

\section{Future directions for solving the conjecture about the Fibonacci-like behavior of $n_g$}

\label{sec:future}
In this section we outline some further thoughts on strong/weak generators and how they might help to solve the Fibonacci conjecture.
First of all, computational evidence suggests that as $g$ grows, the portion of strong generators among all effective generators 
becomes smaller. Namely, the following is conjectured.
\begin{conject}
  Let $S_g$ be the number of all strong generators in all numerical semigroups of genus $g$ and let
  $W_g$ be the number of all weak generators in all numerical semigroups of genus $g$. We conjecture that
  $$
    \lim_{g\to\infty}\frac{S_g}{W_g}=0.
  $$
  \label{ratio-to-0}
\end{conject}

Notice that by Lemma~\ref{lemma:reasonforweakstrong},
if the number of effective generators 
(and so the number of descendants) of a semigroup is $k$
and all $k$ effective generators are weak then the number of 
effective generators (and so the number of descendants) of its descendants
is respectively $0,1,\dots,k-1$.
In \cite{Bras:ngbounds} 
the tree $A$ represented in Figure~\ref{fig:arbreA}
was recursively defined as follows:
Its root is labeled as $1$ and it has a single descendant which is labeled as
$2$. This descendant in turn has two descendants labeled as $1$ and $3$.
At each level $g$, the number of descendants of a node is equal to its label.
From level $g=2$ on, if the label of a node is $k$ then the labels of its descendants
are $0,\dots,k-1$ except for the node with label $k=g+1$, whose descendants have
labels $0,\dots,k-3$,$k-1$,$k+1$.

Because of Lemma~\ref{lemma:reasonforweakstrong}
and because of the particular structure of ordinary semigroups,
the semigroup tree in Figure~\ref{fig:tree}
contains $A$ as a subtree.

Define $A_0=\{1\}$, $A_1=\{2\}$
and for $g\geq 2$
define $A_g$ as 
\begin{eqnarray*}
  A_g&=&\{g+1\}\cup\left(\bigcup_{m\in A_{g-1}}\{0,1,\dots,m-1\}\right)\setminus\{g-2\}.\\
\end{eqnarray*}
The tree $A$
has $A_g$ as the nodes at distance $g$ from its root.
Thus,
$|A_g|\leq n_g$.
It was shown in \cite{Bras:ngbounds} that $|A_g|=2F_g$, 
where $F_i$ denotes the $i$-th Fibonacci number.
From this the lower bound $n_g\geq 2F_{g}$ was deduced.

\begin{figure}
\compatiblegastexun
\setvertexdiam{6}
\letvertex s0=(15,120)
\letvertex s1=(15,100)
\letvertex s21=(0,80)
\letvertex s22=(30,80)
\letvertex aux21=(5,70)
\letvertex aux22=(40,70)
\letvertex s31=(0,60)
\letvertex s32=(20,60)
\letvertex s33=(30,60)
\letvertex s34=(40,60)
\letvertex s35=(50,60)
\letvertex s41=(30,40)
\letvertex s42=(40,40)
\letvertex s43=(50,40)
\letvertex s44=(60,40)
\letvertex s45=(70,40)
\letvertex s46=(80,40)
\letvertex s47=(100,40)
\letvertex s51=(40,20)
\letvertex s52=(60,20)
\letvertex s53=(70,20)
\letvertex s54=(80,20)
\letvertex s55=(90,20)
\letvertex s56=(100,20)
\letvertex s57=(110,20)
\letvertex s58=(120,20)
\letvertex s59=(130,20)
\letvertex s510=(140,20)
\letvertex s511=(170,20)
\letvertex s61=(70,0)
\letvertex s62=(80,0)
\letvertex s63=(90,0)
\letvertex s64=(100,0)
\letvertex s65=(110,0)
\letvertex s66=(120,0)
\letvertex s67=(130,0)
\letvertex s68=(140,0)
\letvertex s69=(150,0)
\letvertex s610=(160,0)
\letvertex s611=(170,0)
\letvertex s612=(180,0)
\letvertex s613=(190,0)
\letvertex s614=(200,0)
\letvertex s615=(210,0)
\letvertex s616=(220,0)
\letvertex s617=(250,0)
%
%
\begin{center}
\resizebox{\textwidth}{!}{
{\Huge{\tt
\begin{picture}(250,120) 
%
\drawvertex(s0){1}
\drawundirectededge(s0,s1){}
\drawvertex(s1){2}
\drawundirectededge(s1,s21){}
\drawundirectededge(s1,s22){}

\drawvertex(s21){1}
\drawvertex(s22){3}
\drawundirectededge(s21,s31){}
\drawundirectededge(s22,s32){}
\drawundirectededge(s22,s34){}
\drawundirectededge(s22,s35){}

\drawvertex(s31){0}
\drawvertex(s32){0}
\drawvertex(s33){1}
\drawvertex(s34){2}

\drawvertex(s33){\makebox(0,0){\resizebox{1.5cm}{!}{\alert{$\times$}}}}
\drawvertex(s35){\alert{4}}

\drawundirectededge(s34,s41){}
\drawundirectededge(s34,s42){}
\drawundirectededge(s35,s43){}
\drawundirectededge(s35,s44){}
\drawundirectededge(s35,s46){}
\drawundirectededge(s35,s47){}

\drawvertex(s41){0}
\drawvertex(s42){1}
\drawvertex(s43){0}
\drawvertex(s44){1}
\drawvertex(s45){2}
\drawvertex(s46){3}

\drawvertex(s45){\makebox(0,0){\resizebox{1.5cm}{!}{\alert{$\times$}}}}
\drawvertex(s47){\alert{5}}

\drawundirectededge(s42,s51){}
\drawundirectededge(s44,s52){}
\drawundirectededge(s46,s53){}
\drawundirectededge(s46,s54){}
\drawundirectededge(s46,s55){}
\drawundirectededge(s47,s56){}
\drawundirectededge(s47,s57){}
\drawundirectededge(s47,s58){}
\drawundirectededge(s47,s510){}
\drawundirectededge(s47,s511){}

\drawvertex(s51){0}
\drawvertex(s52){0}
\drawvertex(s53){0}
\drawvertex(s54){1}
\drawvertex(s55){2}
\drawvertex(s56){0}
\drawvertex(s57){1}
\drawvertex(s58){2}
\drawvertex(s59){3}
\drawvertex(s510){4}

\drawvertex(s59){\makebox(0,0){\resizebox{1.5cm}{!}{\alert{$\times$}}}}
\drawvertex(s511){\alert{6}}

\drawundirectededge(s54,s61){}
\drawundirectededge(s55,s62){}
\drawundirectededge(s55,s63){}
\drawundirectededge(s57,s64){}
\drawundirectededge(s58,s65){}
\drawundirectededge(s58,s66){}
\drawundirectededge(s510,s67){}
\drawundirectededge(s510,s68){}
\drawundirectededge(s510,s69){}
\drawundirectededge(s510,s610){}
\drawundirectededge(s511,s611){}
\drawundirectededge(s511,s612){}
\drawundirectededge(s511,s613){}
\drawundirectededge(s511,s614){}
\drawundirectededge(s511,s616){}
\drawundirectededge(s511,s617){}
\drawvertex(s61){0}
\drawvertex(s62){0}
\drawvertex(s63){1}
\drawvertex(s64){0}
\drawvertex(s65){0}
\drawvertex(s66){1}
\drawvertex(s67){0}
\drawvertex(s68){1}
\drawvertex(s69){2}
\drawvertex(s610){3}
\drawvertex(s611){0}
\drawvertex(s612){1}
\drawvertex(s613){2}
\drawvertex(s614){3}
\drawvertex(s615){4}
\drawvertex(s616){5}

\drawvertex(s615){\makebox(0,0){\resizebox{1.5cm}{!}{\alert{$\times$}}}}
\drawvertex(s617){\alert{7}}

%
%
%
\end{picture}}}}\end{center}
$$\vdots$$
\caption{Tree $A$. It is a subtree of the tree of numerical semigroups.}
\label{fig:arbreA}
\end{figure}
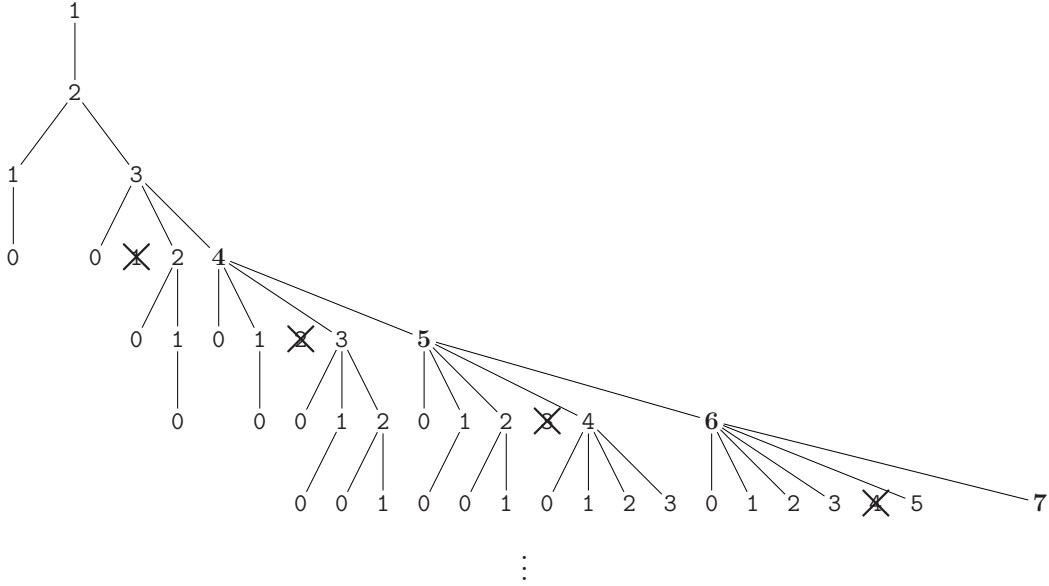

The next Proposition observes that no matter how a tree behaves at the beginning,
if at some point its generation rule coincides with the one of $A$, the Fibonacci behavior is observed from some point on.

\begin{prop}
  Let $l\geq 2$ be an integer and let $L_l$ be a multiset composed of some (maybe with repetitions) numbers $\le l-2$, and numbers $l-1$ and 
  $l+1$.
  For $k>l$ define recursively
  $$
    L_k=\{k+1\}\bigcup\Big(\bigcup_{m\in L_{k-1}} \{0,1,\dots,m-1\}\Big)\setminus\{k-2\}.
  $$
  Then, for all $k\ge 2l$:
  $$
    |L_k|=|L_{k-1}|+|L_{k-2}|.
  $$
Even more: $|L_k|=2F_k$.
  \label{start-from-arbitrary}
\end{prop}
\begin{proof}
In \cite{Bras:ngbounds} it is proven that for $l=2$ and $L_2=\{1,3\}$,
the recursively defined sets $L_k$ satisfy $|L_k|=2F_k$ for all $k\geq 2$.
This proves the lemma in the particular case in which
$l=2$ and $L_2=\{1,3\}$.

Next we will prove that if $l$, $l'$ are integers and the multisets 
$L_l$, $L'_{l'}$ satisfy the hypothesis, then 
$L_k=L'_k$ for all $k\geq \max(2l,2l')$. This, together with the result in
\cite{Bras:ngbounds} will end the proof.

Suppose $m\in L_s$, $m\neq s+1$. Then $m$ gives rise to a subset $\{0,\dots,m-1\}\subseteq L_{s+1}$ and to a subset in $L_{s+2}$ whose maximum 
element is $m-2$ and 
to a subset in $L_{s+3}$ whose maximum element is $m-3$ and so on. 
However, the fact that $m\in L_s$ does not affect 
$L_{s'}$ for $s'>k+m$.
Similarly, the only element in $L_l$ that affects $L_k$ for any $k\geq 2l$ is $l+1$. Consequently, $L_k=L'_k$ for any $k\geq 2l$.
\end{proof}

A rough idea of future approaches to the Fibonacci problem would be: observe that the number of strong generators becomes negligible compared 
to all effective generators as $g\to\infty$, then the semigroup tree behaves more and more like the tree $A$ from \cite{Bras:ngbounds}.
So roughly speaking we are in the situation of Proposition \ref{start-from-arbitrary}. Pushing this idea further could help to solve the 
Fibonacci conjecture.

Finally, we would like to mention some computational evidence that suggests that strong generators appear quite regularly.
Let $n^i_g$ be the number of numerical semigroups of genus $g$ with $i$ strong generators. Then we conjecture that

$$
  n^i_g= 0 \mbox{ for } i>\Big\lfloor\frac{g-1}{2}\Big\rfloor.
$$

It is observed that as $g$ increases, $n^{\lfloor\frac{g-1}{2}\rfloor-j}_g$ 
approaches a constant for $g$ even and another constant for $g$ odd. 
So, we can define
two sequences 
$$
  \begin{array}{l}
   e_j=\lim_{k \to \infty} n^{k-1-j}_{2k},\\
   o_j=\lim_{k \to \infty} n^{k-j}_{2k+1}.
  \end{array}
$$

The first terms of the sequence $e$ have been observed to be 
$$2,2,5,12,21,45.$$
And the first terms of the sequence $o$ have been observed to be
$$1,2,3,8,14,34-35.$$
It seems that $e_j\geq \sum_{l=0}^{j-1}e_l$ and the same for $o$, so we conjecture in particular that the $e$- and $o$-sequences are 
superincreasing.

\section{Conclusions}
In this paper we went a step further on the study of the structure of the semigroup tree.
Namely we described the nodes that correspond to some well-studied
classes of numerical semigroups, like symmetric, pseudosymmetric and Arf. Apart from this we also considered what kind of chains appear in the
semigroup tree. Namely, when a node (semigroup) belongs to an infinite chain, and when the number of such chains is finite/infinite. We
concluded the paper with some conjectures and observations regarding the number of strong generators. These conjectures hopefully can help
in tackling the Fibonacci problem.

\section*{Acknowledgement}
This work was partly supported by the Spanish Ministry of Education
through projects TSI2007-65406-C03-01 "E-AEGIS" and
CONSOLIDER CSD2007-00004 "ARES",
and by the Government of Catalonia under grant 2005 SGR 00446. The second author was partially funded by the DASMOD Cluster of Excellence
in Pheinland-Palatinate.
The second author would like to thank his Ph.\,D. supervisor Prof.\,Dr. Gert-Martin Greuel and his second supervisor Prof.\,Dr. Gerhard Pfister 
for continuous support and encouragement.


\end{document}